\documentclass[12pt]{article}% Math and Physical Sciences Reference Style

\usepackage{authblk} % \affil
\usepackage{hyperref} % ref in PDF
\usepackage{dsfont}
\usepackage[T1]{fontenc} 
\usepackage{amsmath}
\usepackage{amsfonts,amssymb}
\usepackage{extarrows}
\usepackage{geometry}
\usepackage{enumerate}
\usepackage{enumitem}
\usepackage{color}

\usepackage{amsthm}

\usepackage[utf8]{inputenc}
\def\0{\emptyset}

\newtheorem{theorem}{Theorem}[section]

\newtheorem{lemma}[theorem]{Lemma}
\newtheorem{claim}[theorem]{Claim}

\newtheorem{corollary}[theorem]{Corollary}

\newtheorem{conjecture}[theorem]{Conjecture}

\usepackage{graphicx}

\DeclareMathOperator{\ex}{\mathrm{ex}}
\DeclareMathOperator{\Tr}{\mathrm{Tr}}

\usepackage{enumerate}

\usepackage{
	amsmath,			% Math Environments
	amssymb,			% Extended Symbols
	enumerate,		    % Enumerate Environments
	graphicx,			% Include Images
	lastpage,			% R\part{title}eference Lastpage
	multicol,			% Use Multi-columns
	multirow,			% Use Multi-rows
	pifont,			    % For Checkmarks
}

\usepackage[numbers]{natbib}
\begin{document}

% --- PAPER INFO ---

\title{Forbidding matching as trace in uniform hypergraphs}
% \author{
%     {\small\bf Yichen Wang}\thanks{email:  wangyich22@mails.tsinghua.edu.cn}\quad 
%     % {\small\bf Mengyu Cao}\thanks{email:  myucao@ruc.edu.cn}\quad 
%     % {\small\bf Zequn Lv}\thanks{email:  lvzq19@mails.tsinghua.edu.cn}\quad 
%     {\small\bf Mei Lu}\thanks{email: lumei@tsinghua.edu.cn}\\
%     {\small Department of Mathematical Sciences, Tsinghua University, Beijing 100084, China.}\\
%     % {\small Institute for Mathematical Sciences, Renmin University of China, Beijing 100086, China.}\\
% }

% \author[1]{\small\bf Yichen Wang\thanks{E-mail:  wangyich22@mails.tsinghua.edu.cn}}
% \author[2]{\small\bf Mengyu Cao\thanks{E-mail:  myucao@ruc.edu.cn}}
% \author[1]{\small\bf Zequn Lv\thanks{E-mail:  lvzq19@mails.tsinghua.edu.cn}}
% \author[1]{\small\bf Mei Lu\thanks{E-mail:  lumei@tsinghua.edu.cn}}

%     {\quad
%     {\small\bf Mengyu Cao}\thanks{email:  myucao@ruc.edu.cn}\quad
%     {\small\bf Zequn Lv}\thanks{email:  lvzq19@mails.tsinghua.edu.cn}\quad
%     {\small\bf Mei Lu}\thanks{email: lumei@tsinghua.edu.cn}\\
% }

\author[1]{\small\bf Yichen Wang\thanks{E-mail: wangyich22@mails.tsinghua.edu.cn}}
\author[3]{\small\bf Xin Cheng\thanks{Corresponding author: E-mail: xincheng@mail.nwpu.edu.cn}}
\author[2]{\small\bf Ervin Gy\H{o}ri\thanks{E-mail:  gyori.ervin@renyi.hu}}
% \author[2]{\small\bf Kitti Varga\thanks{E-mail: vkitti@cs.bme.hu}}
% \author[4,5]{\small\bf Yuanpei Wang\thanks{E-mail: boyuan@shu.edu.cn}}
\author[1]{\small\bf Xiamiao Zhao\thanks{E-mail: zxm23@mails.tsinghua.edu.cn}}
% \author[4,5]{\small\bf Junpeng Zhou\thanks{E-mail: junpengzhou@shu.edu.cn}}

\affil[1]{\small Department of Mathematical Sciences, Tsinghua University, Beijing, P.R. China.}
\affil[2]{\small HUN-REN Alfréd Rényi Institute of Mathematics, Budapest, Hungary.}
\affil[3]{\small School of Mathematics and Statistics, Northwestern Polytechnical University and Xi'an-Budapest Joint Research Center for Combinatorics, Xi'an, Shaanxi, P.R. China.}

% \affil[4]{\small Department of Mathematics, Shanghai University, Shanghai 200444, P.R. China.}
% \affil[5]{\small Newtouch Center for Mathematics of Shanghai University, Shanghai 200444, P.R. China.}

%\author{Yichen Wang}
\date{}
% \date{2026.01.11}

\maketitle

\begin{abstract}
We say a hypergraph $\mathcal{H}$ contains a hypergraph $\mathcal{G}$ as trace if there exists a vertex subset $S \subseteq V(\mathcal{H})$ such that $|S| = |V(\mathcal{G})|$ and $\{e \cap S: e \in E(\mathcal{H})\}$ contains $\mathcal{G}$ as a sub-hypergraph.
We use $\ex_r(n, \Tr_r(\mathcal{G}))$ to denote the maximum number of hyperedges in an $r$-uniform hypergraph on $n$ vertices not containing $\mathcal{G}$ as a trace.
The study of Tur\'{a}n numbers for traces was initiated by Mubayi and Zhao who studied the case when $\mathcal{G}$ is a complete graph.

Let $M_{s+1}$ denote the graph of a matching with $s+1$ edges. In this paper, we give the upper bound of $\ex_r(n, \Tr_r(M_{s+1}))$ which is sharp asymptotically.
When $r=3$, we give the exact value of $\ex_3 (n, \Tr_3 (M_{s+1}))$.
We also consider the generalized Tur\'{a}n number in the case of matching.
That is, the maximum number of copies of clique $\mathcal{K}_t^r$ in hypergraphs forbidding $\Tr_r (M_{s+1})$ as a trace.
We give an upper bound which is sharp asymptotically and when $r=3$, we give the exact value.
The Tur\'{a}n number of forbidding a matching and the other graph is another well studied topic initiated by Alon and Frankl.
We also consider an analogue problem for the trace version, i.e., forbidding trace of matching and trace of complete graph as subgraphs.
\end{abstract}

%\textbf{AMS classification: }\textit{05C75, 05C65, 05C05}\vskip 0.3cm

% {\bf Keywords:}  strong edge coloring; strong chromatic index; treewidth.
\vskip.3cm

% -------------------------------------------------------------
% --------- Here begins INTRODUCTION SECTION ------------------
\section{Introduction}

In this paper, we use capital letters, say $G$ and $F$, to denote graphs, and calligraphic letters, say $\mathcal{G}$ and $\mathcal{H}$, to denote hypergraphs.
We use the notation $O(1)$ to denote a term that is at most $C$ where $C$ is a constant.
Especially, when we have constants $c_1,c_2,\ldots,c_k$, we use the notation $O_{c_1,c_2,\ldots,c_k}(1)$ to denote a term that is at most $C(c_1,c_2,\ldots,c_k)$ where $C(c_1,c_2,\ldots,c_k)$ is a constant depending on $c_1,c_2,\ldots,c_k$.
% We use the notation of $O(.)$ in the general sense.
We use $v(G) = |V(G)|$ and $e(G) = |E(G)|$ to denote the number of vertices and edges of $G$, respectively. 
The same notation applies to hypergraphs.
A hypergraph is $r$-uniform if all hyperedges have size $r$.
Let $\{\mathcal{F}\}_{i \in I}$ be a family of $r$-uniform hypergraphs.
The Tur\'{a}n number $\ex_r (n, \{\mathcal{F}\}_{i \in I})$ denotes the maximum number of hyperedges in an $r$-uniform hypergraph on $n$ vertices containing no member of $\{\mathcal{F}\}_{i \in I}$ as a subgraph.

For a hypergraph $\mathcal{H}$, the \textbf{trace} of $\mathcal{H}$ on a vertex subset $S$ is the hypergraph $\mathcal{H}_{|S} = \{e \cap S: e \in E(\mathcal{H})\}$.
Note that $\mathcal{H}_{|S}$ is not necessarily uniform even if $\mathcal{H}$ is uniform.
Let $\mathcal{H}$ be an $r$-uniform hypergraph and $\mathcal{G}$ be a hypergraph. We say $\mathcal{H}$ contains $\mathcal{G}$ as a trace if there exists a vertex subset $S$ of $\mathcal{H}$ with $|S| = v(\mathcal{G})$ such that $\mathcal{H}_{|S}$ contains $\mathcal{G}$ as a subgraph.
We use $\Tr_r(\mathcal{G})$ to denote the family of minimal $r$-uniform hypergraphs containing $\mathcal{G}$ as a trace.
Note that, a hypergraph contains $\mathcal{G}$ as a trace if and only if it contains some member of $\Tr_r(\mathcal{G})$ as a sub-hypergraph.
% In the following, when $\mathcal{G}$ is a graph, we use $F$ to denote $\mathcal{G}$.

Denote the complete $r$-uniform hypergraph on $t$ vertices by $\mathcal{K}_t^{r}$.
The studies about forbidding traces of hypergraphs were initiated by Mubayi and Zhao~\cite{mubayi2007forbidding} on the trace of $\mathcal{K}_t^{r}$.
In the same paper, they proposed a conjecture of $\ex_r(n, \Tr_r(\mathcal{K}_t^{r'}))$ and proved the asymptotics for some cases.
% The studies about Tur\'{a}n number of traces of specific graphs, can be found in~\cite{FUREDI2023103692, 2023arXiv231005601G, luo2021forbidding, 2022arXiv220605884Q}.
% Also, the readers may refer to a recent survey about forbidding trace of graphs~\cite{2025arXiv250723375L}.

In this paper, we consider $\Tr_r (\mathcal{F})$ where $\mathcal{F}$ is a graph, denoted by $F$ in the following.
When $F$ is a graph, $\Tr_r (F)$ is called graph-based hypergraphs.
There are several ways to define graph-based hypergraphs.
The first way is the \textbf{expansion} of a graph $F$, denoted by $F^{(r)+}$, which is obtained from $F$ by enlarging each edge of $F$ with a set of $r-2$ new vertices such that all these new sets are disjoint.
Secondly, the \textbf{Berge} copies of a graph $F$, denoted by $\mathrm{B}_r (F)$, is a family of $r$-uniform hypergraphs obtained by replacing each edge of $F$ with a distinct $r$-edge containing it.
Notice that $F^{(r)+} \in \Tr_r (F) \subseteq \mathrm{B}_r (F)$.
Thus we have
\begin{equation*}
    \ex_r (n, \mathrm{B}_r (F)) \leq \ex_r (n, \Tr_r (F)) \leq \ex_r (n, F^{(r)+}).
\end{equation*}

The order of magnitude of $\mathrm{ex}_r(n, \Tr_r(F))$ was determined by Füredi and Luo~\cite{FUREDI2023103692}, in the following theorem.
\begin{theorem}[Füredi and Luo~\cite{FUREDI2023103692}]\label{thm: furedi luo}
    For any graph $F$ with at least one edge, as $n$ tends to infinity, we have
    \[
        \mathrm{ex}_r(n, \Tr_r(F)) = \Theta( \max_{2 \le s \le r} \{ \mathrm{ex}(n,K_s,F) \}).
    \]
\end{theorem}

There are several studies about $\ex_r(n, \Tr_r(F))$ for other graphs $F$.
The study when $F$ is an outerplanar graph or a star can be found in~\cite{FUREDI2023103692, 2022arXiv220605884Q}. 
When $r=3$ and $F=K_{2,t}$, it is studied by Luo and Spiro~\cite{luo2021forbidding}, and Qian and Ge~\cite{2022arXiv220605884Q}.
When $r=3$ and $F = C_4$, it is studied by Luo and Spiro~\cite{luo2021forbidding}, and Gerbner and Picollelli~\cite{2023arXiv231005601G}.
When $r=3$ and $F = K_{1,1,t}$, it is studied by Gerbner and Picollelli~\cite{2023arXiv231005601G}.
The non-degenerate case when $r < \chi(F)$ follows from the expansion version~\cite{gerbner2025}.
Also, the readers may refer to a recent survey about forbidding trace of graphs~\cite{2025arXiv250723375L}.

Let $M_{s+1}$ denote the matching of size $s+1$, i.e., the graph consisting of $s+1$ independent edges.
In 2022, Khormali and Palmer \cite{KHORMALI2022103506} completely determined the Tur\'{a}n number of Berge matchings for sufficiently large $n$.
Then, Kang, Ni, and Shan \cite{KANG2022112901} determined $\ex_r(n, B_r(M_{s+1}))$ for some cases.
Later, Wang, Yang, Zhao, Bai, and Zhou \cite{zhao2025turantypeproblemsbergematchings} solved the remaining cases.
Hence, $\ex_r(n, B_r(M_{s+1}))$ is determined for every $n$ and, especially, when $n$ is large we have $\ex_r(n, B_r(M_{s+1})) = \binom{s}{r-1}(n-s) +\binom{s}{r}$.

For the Tur\'{a}n number of expansion of matching, Erd\H{o}s \cite{erdos1965problem} proposed the Erd\H{o}s matching conjecture.
It is known to hold for sufficiently large $n$.
The most recent result is due to Frankl \cite{FRANKL20131068}, who proved that $\ex_r (n, M_{s+1}^{(r)+}) = \sum_{i=1}^{s} \binom{s}{i} \binom{n-s}{r-i}$ for $r, s \geq 1$ and $n \geq (2s+1)r - s$.

In this paper, we study the Tur\'{a}n number of forbidding a matching $M_{s+1}$ as a trace.
Before stating our main theorems, we need some definitions.
Let $\mathcal{H}$ be an $r$-uniform hypergraph. We say a set $D$ of vertices is a \textbf{dominating set} of $\mathcal{H}$ if for every vertex $v \notin D$, there exists a hyperedge $e$ such that $v \in e$ and $e\setminus \{v\} \subseteq D$.
The \textbf{domination number} of a hypergraph $\mathcal{H}$, denoted by $\gamma(\mathcal{H})$, is the minimum size of a dominating set in $\mathcal{H}$.
There are many variants of domination in hypergraphs.
Our version is introduced by Divya, Ramakrishnan, and Arumugam in~\cite{divya2024new}.
We say a set $S$ of vertices in an $r$-uniform hypergraph $\mathcal{H}$ is a \textbf{dominated set} if its complement is a dominating set.
The \textbf{dominated number} of a hypergraph $\mathcal{H}$, denoted by $\phi(\mathcal{H})$, is the maximum size of a dominated set in $\mathcal{H}$.
When $r=2$, the dominated set is also known as enclaveless set (see \cite{slater1977enclaveless}).
The investigation of dominating sets in graphs can be found in \cite{haynes2013fundamentals, mojdeh2019outer, rad2019complexity}.
By the definitions of dominating set and dominated set, we have
\begin{equation}\label{eq: dominated and domination}
\phi(\mathcal{H}) + \gamma(\mathcal{H}) = v(\mathcal{H}).
\end{equation}

Now we state our main results.
Let $f_r(s)$ be the maximum size of an $r$-uniform hypergraph with dominated number at most $s$.

\begin{theorem}\label{thm: general r}
    For $s \ge 1$, when $n$ is sufficiently large we have
    \[
        \ex_r(n, \Tr_r(M_{s+1})) = f_{r-1}(s)\cdot n + O_{r,s}(1).
    \]
    % For a $r$-uniform hypergraph $\mathcal{H}$ on $n$ vertices forbidding $M_{s+1}$ as a trace, $s \ge 1$, we have
    % \[
    %     e(\mathcal{H}) \le f_{r-1}(s)\cdot n + O_{r,s}(1),
    % \]
    % when $n$ is sufficiently large.
    A construction achieving the asymptotic bound is as follows.
    Let $\mathcal{G}$ be the $(r-1)$-uniform hypergraph with $f_{r-1} (s)$ hyperedges and $\phi(\mathcal{G}) \leq s$ without isolated vertices. 
    For each of the remaining vertices $v$, let $v \cup e$ be a hyperedge for every $e \in E(\mathcal{G})$.
\end{theorem}

In particular, we give the exact Tur\'{a}n number of $\Tr_r(M_{s+1})$ when $r=3$.

\begin{theorem}\label{thm: r=3 case}
	For $s \ge 1$, when $n$ is sufficiently large, we have
	\[
		\ex_3(n, \Tr_3(M_{s+1})) = \left\lfloor \frac{s(s+2)}{2} \right\rfloor (n-s-2) + \binom{s+2}{3}.
	\]
	The equality holds only when $\mathcal{H}$ is constructed as follows:
	Let $G$ be the graph obtained by removing a minimum edge cover from the complete graph $K_{s+2}$.
	For each of the remaining $n-(s+2)$ vertices $v$, let $v \cup e$ be a hyperedge for every $e \in E(G)$.
	Then add hyperedges for every triple contained in $V(G)$.
\end{theorem}
 
% It is worth mentioning that $\ex_3(n, \Tr_3(M_{s+1}))$ does not equal to $\ex_3(n, \mathrm{B}_3(M_{s+1}))$ or $\ex_3(n, M_{s+1}^{(3)+})$.

The generalized Tur\'{a}n number is also studied in hypergraphs.
That is, instead of counting hyperedges, we want to maximize the number of copies of a hypergraph like cliques.
The following theorem gives the upper bound of the number of copies of cliques in hypergraphs forbidding $\Tr_r(M_{s+1})$ as a trace.
Let $g_r(s, t)$ be the maximum number of copies of $\mathcal{K}^r_{t}$ in an $r$-uniform hypergraph with dominated number at most $s$ without isolated vertices.

\begin{theorem}\label{thm: generalized Turan number}
    For an $r$-uniform hypergraph $\mathcal{H}$ on $n$ vertices forbidding $M_{s+1}$ as a trace, $r + 1 \leq t$, let $k^r_t(\mathcal{H})$ be the number of copies of $\mathcal{K}^r_{t}$ in $\mathcal{H}$, we have
    \[
        k^r_t(\mathcal{H}) \le g_{r-1}(s, t-1) \cdot (n-(s+r-2)) + O_{s,r}(1),
    \]
    when $n$ is sufficiently large.
    Especially, when $t \ge s+r$, $k^r_t(\mathcal{H}) \le O_{s,r}(1)$.
    When $t \le s+r-1$, a construction achieving the asymptotic bound is as follows.
    Let $\mathcal{G}$ be the complete hypergraph $\mathcal{K}_{s+r-2}^{r-1}$. For each of the remaining $n-(s+r-2)$ vertices, let $v \cup e$ be a hyperedge for every $e \in E(\mathcal{G})$. 
    Then for each $\mathcal{K}_{t-1}^{r-1}$ in $\mathcal{G}$, add hyperedges for every $r$-sets contained in the vertex set of the $\mathcal{K}_{t-1}^{r-1}$.
\end{theorem}

When $r=3$, we have the exact result.
\begin{theorem}\label{thm: generalized Turan number r=3}
    For a $3$-uniform hypergraph $\mathcal{H}$ on $n$ vertices forbidding $M_{s+1}$ as a trace, $4 \leq t \le s$, let $k^3_t(\mathcal{H})$ be the number of copies of $\mathcal{K}^3_{t}$ in $\mathcal{H}$, we have
    \[
        k^3_t(\mathcal{H}) \le \binom{s+1}{t-1} \cdot (n-s-1) + \binom{s+1}{t},
    \]
    when $n$ is sufficiently large.
    A construction achieving the asymptotic bound is as follows.
    Let $G$ be the complete graph on $s+1$ vertices $K_{s+1}$.
    For each of the remaining $n-(s+1)$ vertices, let $v \cup e$ be a hyperedge for every $e \in E(G)$.
    Then add hyperedges for every $3$-set contained in $V(G)$.
\end{theorem}

In 2024, Alon and Frankl~\cite{ALON2024223} initiated the study of Tur\'{a}n problems on graphs forbidding a matching and the other graph.
Later, several results about Tur\'{a}n problems on hypergraphs with bounded matching number were given by several different groups of scholars, see~\cite{chen2025triplesystemsboundedmatching, GERBNER2025104155,  yang2025hypergraphanaloguealonfrankltheorem}.
Inspired by their work, we also consider the Tur\'{a}n number of forbidding a matching as a trace together with a clique as a trace.

Given a $r$-uniform hypergraph $\mathcal{G}$, a graph $F$ and a vertex set $W$ of $F$, we say there is a dominated copy of $(F,W)$ in $\mathcal{G}$ if there exists a bijection $\tau$ from $V(F)$ to a subset $S$ of $V(\mathcal{G})$ such that the following conditions hold:
\begin{enumerate}
    \item For every edge $uv \in E(F)$, there exists a hyperedge in $\mathcal{G}$ intersecting $S$ at exactly two vertices $\tau(u)$ and $\tau(v)$.
    \item For every vertex $w \in W$, there exists a hyperedge in $\mathcal{G}$ contianing $\tau(w)$ and intersecting $S$ at exactly $\tau(w)$.
\end{enumerate}

For a vertex set $I$, let $N(I)$ be the set of vertices in $F$ adjacent to at least one vertex in $I$.
% Given a graph $F$, let $\mathcal{D}(F)$ be the set of pairs $\{(F-I, N(I)): \text{$I$ is an independent set in $F$}\}$.
Let $h_r(s, F)$ be the maximum number of hyperedges in an $r$-uniform hypergraph with $\phi(\mathcal{G}) \le s$ does not contain a dominated copy of $(F-I,N(I))$ for any independent set $I$ in $F$.

\begin{theorem}\label{thm: matching and a graph}
    For $s \ge 1$, when $n$ is sufficiently large, we have
	\[
		\ex_r(n, \{\Tr_r(M_{s+1}), \Tr_r(F)\}) = h_{r-1}(s, F) \cdot n + O_{r,s,v(F)}(1).
	\]
    A construction which achieves the asymptotic bound is as follows.
    Let $\mathcal{G}$ be the $(r-1)$-uniform hypergraph with $h_{r-1}(s, F)$ hyperedges and $\phi(\mathcal{G}) \leq s$ without isolated vertices. 
    For each of the remaining vertices $v$, let $v \cup e$ be a hyperedge for every $e \in E(\mathcal{G})$.
\end{theorem}

The paper is organized as follows:
In Section~\ref{sec: preliminary}, we present several lemmas that will be used in the proofs of main theorems.
In Section~\ref{sec: structure}, we give the structure of $\Tr_r (M_{s+1})$-free hypergraph, which is the key to proving our main theorems.
In Section~\ref{sec: main proofs}, we give the proofs of main theorems.
In Section~\ref{sec: concluding remarks}, we give some concluding remarks.

% -------------------------------------------------------------
\section{Preliminaries}\label{sec: preliminary}

% As defined by Alan Goldman and introduced by Slater in~\cite{slater1977enclaveless}, for a subset $S$ of vertices in a graph $G$, a vertex $v \in S$ is an \textbf{enclave} if it and all of its neighbors are also in $S$, that is, $N[v] \subseteq S$.
% A set $S$ is an \textbf{enclaveless set} if it does not contain any enclaves.
% The \textbf{enclaveless number} of a graph $G$, denoted by $\phi(G)$, is the maximum size of an enclaveless set in $G$.

% We say a set $D$ of vertices in a graph $G$ is a \textbf{domination} set if every vertex not in $D$ has a neighbor in $D$.
% The \textbf{domination number} of a graph $G$, denoted by $\gamma(G)$, is the minimum size of a domination set in $G$.
% By observation, if $S$ is an enclaveless set in $G$, then $V(G)\setminus S$ is a domination set.
% Therefore, we have
% \begin{equation}\label{eq: enclaveless and domination}
% \phi(G) + \gamma(G) = v(G).
% \end{equation}

In this section, we present several lemmas that will be used in the proofs of the main theorems.

\begin{theorem}[Lov\'asz version of Kruskal-Katona Theorem~\cite{lovasz1993combinatorial}]\label{thm: Kruskal-Katona Theorem}
    Let $G$ be a graph with at least $\binom{x}{r}$ copies of $K_r$ for some integer $x$, then $G$ has at least $\binom{x}{r'}$ copies of $K_r'$ for some integer $2 \le r' < r$.
    In particular, when $r' = 2$, $G$ has at least $\binom{x}{2}$ edges.
    Moreover, the equality holds only when $G$ is the union of a clique $K_x$ and some isolated vertices.
\end{theorem}

Actually, they stated a more general result when $x$ is a real number.
However, we only need the integer version in our proofs.
Vizing~\cite{vizing1965estimate} determined the maximum number of edges in a graph with given domination number.
% For a graph $G$ with $n$ vertices and domination number $k$, he showed that the maximum number of edges is $\lfloor (n-k+2)(n-k)/2 \rfloor$ provided $k \ge 2$.

\begin{theorem}[Vizing~\cite{vizing1965estimate}]\label{thm: vizing}
	For a graph $G$ with $n$ vertices and domination number $k$, the maximum number of edges is $\lfloor (n-k+2)(n-k)/2 \rfloor$ provided $k \ge 2$.
	The maximum is uniquely attained by the graphs constructed as follows: take the complete graph $K_{n-k+2}$,  remove a minimum edge cover and then add $k-2$ isolated vertices.
\end{theorem}

Combining (\ref{eq: dominated and domination}) and Theorem~\ref{thm: vizing}, we have the following corollary.

\begin{corollary}\label{cor: vizing and phi}
	For a graph $G$ with $\phi(G) \le s$, we have $e(G) \le \lfloor s(s+2)/2 \rfloor$.
	The equality holds only when $G$ is the graph obtained by removing a minimum edge cover from the complete graph $K_{s+2}$ and then adding some isolated vertices.
\end{corollary}

Later, Fulman \cite{FULMAN1994403} generalized Vizing's result so that it takes  the graph's maximum degree into consideration.

\begin{theorem}[Fulman \cite{FULMAN1994403}]\label{thm: Fulman domination number}
    Let $G$ be a graph on $n$ vertices with domination number $k$ and maximum degree $\Delta$.
    Then $e(G) \leq \lfloor \frac{(n - k)(n - k +2) - \Delta (n - k - \Delta)}{2} \rfloor$.
\end{theorem}

We say a graph $G$ is \textbf{domination-critical} if either it is a complete graph, or if the graph obtained by joining any two nonadjacent vertices by an edge has domination number equal to $\gamma(G) - 1$.
For two nonadjacent vertices $u$ and $v$, the identification of $u$ and $v$ is replacing $u$ and $v$ by a new vertex $w$ which is adjacent to all the neighbors of $u$ and $v$. Then we have the following lemma by Vizing~\cite{vizing1965estimate}.

\begin{lemma}[Vizing \cite{vizing1965estimate}]\label{lem: domination number after identifying}
    If a graph $H$ is obtained from a domination-critical graph $G$ by identifying two nonadjacent vertices, then $\gamma(H) = \gamma(G) - 1$.
\end{lemma}

A natural question is to determine the maximum number of copies of a clique in a graph with given domination number.
With the help of above lemmas, we can prove the following lemma.

\begin{lemma}\label{lem: number of clique}
    Let $k_t(G)$ be the number of copies of $K_t$ in a graph $G$.
    For a graph $G$ on $n$ ($n \geq k+t-1$) vertices with $\gamma(G) \geq k$, then $k_t(G) \leq \binom{n-k+1}{t}$.
    The equality holds only when graph $G$ is the union of a clique $K_{n-k+1}$ and $k-1$ isolated vertices.
\end{lemma}
\begin{proof}
    % Let $G$ be a complete graph $K_{n-k+1}$ union $k-1$ isolate vertices.
    % It is easy to see that the domination number of $G$ is $k$ and there are $\binom{n-k+1}{t}$ copies of $K_t$ in $G$.
    Let $G$ be a graph on $n$ vertices with $\gamma(G) \geq k$ maximizing $k_t(G)$.
    If there are several candidates, we choose the graph with the maximum number of edges.
     Then the graph $G$ is domination-critical and $\gamma(G) = k$.
     It is clear that $\Delta(G) \le n-k$, otherwise, we could find a dominating set of size less than $k$.

     Case 1. If $\Delta(G) \leq n-k-1$, then we have
     \begin{equation*}
         e(G) \leq \left\lfloor \frac{(n - k)(n - k +2) - \Delta (n - k - \Delta)}{2} \right\rfloor \leq \left\lfloor \binom{n-k+1}{2} + \frac{1}{2} \right\rfloor = \binom{n-k+1}{2},
     \end{equation*}
     by Theorem~\ref{thm: Fulman domination number}.
    By Theorem~\ref{thm: Kruskal-Katona Theorem}, we have $k_t(G) \leq \binom{n-k+1}{t}$ and the equality holds only when $G$ is the union of a clique $K_{n-k+1}$ and some isolated vertices.

     Case 2. If $\Delta(G) = n-k$. We prove it by induction on $k$. When $k=1$, $G$ is a complete graph and we are done.

     \begin{figure}[t]
        \centering
        \includegraphics[width=0.5\textwidth]{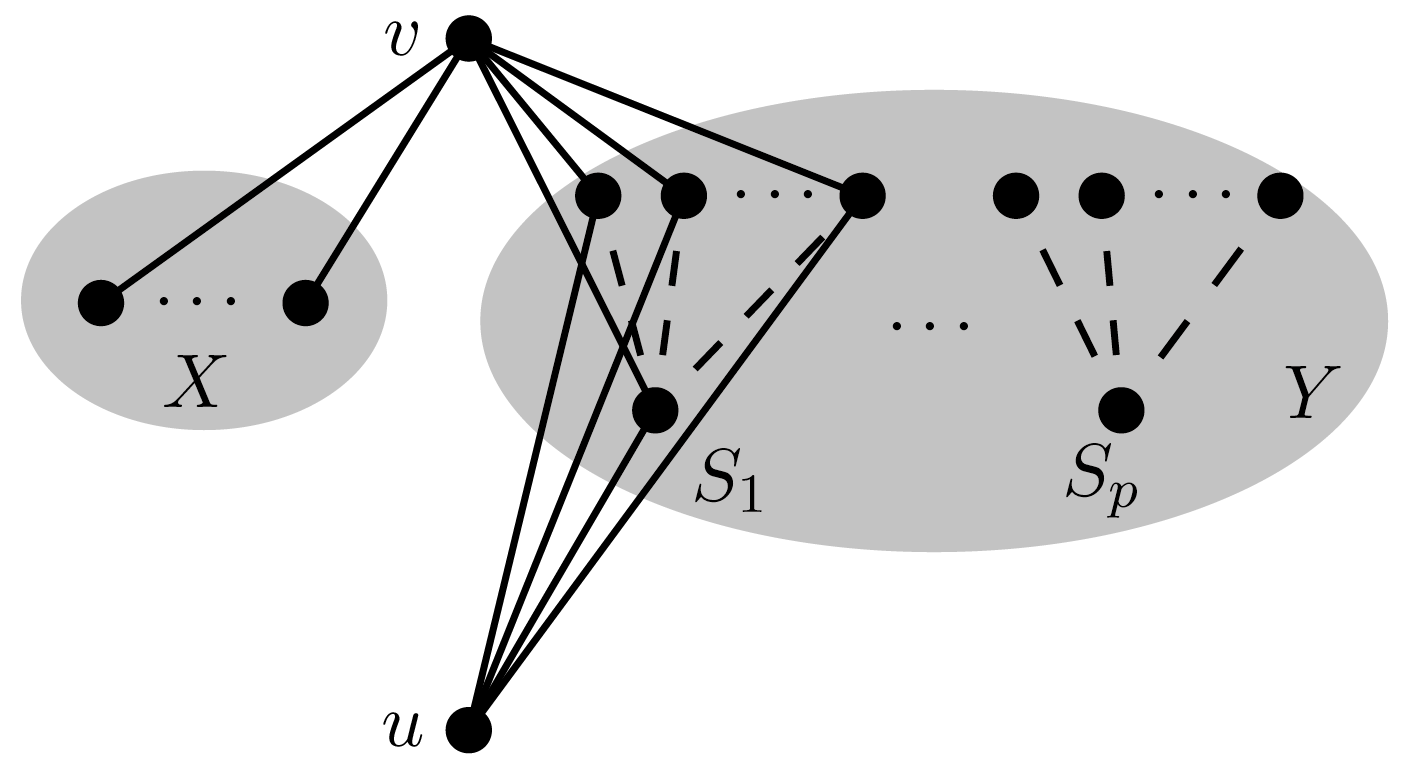}
        \caption{The dotted lines are the missing edges in $Y$.}
        \label{fig: graph}
     \end{figure}

     When $k = 2$, the condition $\gamma(G) \ge 2$ is equivalent to $\Delta \le n - 2$.
     We first prove the case when $t = 3$.
     Let $v$ be the vertex with maximum degree $\Delta = n-k$.
     Let $U = V(G) \setminus (N(v) \cup \{v\})$, then $U$ contains only one vertex, say $u$.
     Let $X$ be the set of vertices in $N(v)$ adjacent to all other vertices in $N[v]$. 
     Since $\Delta \le n-2$, for every $x \in X$, $ux$ is not an edge.
     Let $Y = N(v) \setminus X$.
     If $Y = \emptyset$, then $N[v]$ is a $K_{n-1}$ and we are done.
     Then we may assume $|Y| \ge 2$.
     We first claim that for every $y \in Y$, $uy$ is an edge.
     Otherwise, we can add the edge $uy$ to $G$ to get a graph with more copies of $K_3$ or more edges maintaining $\Delta \le n-2$.
     Let $M$ be the set of missing edges in $Y$, i.e., the non-edges in $Y$.
     Then $M$ covers all vertices in $Y$ by the definition.
     We claim that, for every $e \in M$, one of its endpoints has degree one in $M$.
     Otherwise, we can add the edge $e$ to $G$ to get a graph with more copies of $K_3$ or more edges maintaining $\Delta \le n-2$.
     So $M$ must be the union of disjoint stars.
     Assume $M$ is the union of $p$ stars, $S_1, S_2, \ldots, S_p$ where $S_i$ has $s_i$ vertices~(see Figure~\ref{fig: graph}).
     Then $\sum_{i=1}^{p} s_i = |Y|$ and $|M| = |Y| - p$.
     The number of triangles in $G$ is 
     \begin{equation*}
     \begin{aligned}
        & \binom{n-2}{2} + \binom{|Y| }{2} - 2\left( |Y| - p \right) + \binom{n-2}{3} - (|Y| - p) (n-4) + \sum_{i=1}^{p} \binom{s_i-1}{2}   \\
        = & \binom{n-1}{3} + \binom{|Y|}{2} - (|Y| - p) (n-2) + \sum_{i=1}^{p} \binom{s_i-1}{2}.
     \end{aligned}
     \end{equation*}
     
     Note that $s_i \ge 2$ for each $i$ and $\sum_{i=1}^{p} (s_i-1) = |Y| - p$, then $\sum_{i=1}^{p} \binom{s_i-1}{2} \le \binom{|Y|-2p+1}{2}$.
     Combining $|Y| \le n-2$ and $2p \le |Y|$, we have the number of triangles is at most $\binom{n-1}{3}$, and the equality holds only when $M = \emptyset$ and $Y = \emptyset$.
     Then $k_3(G) \le \binom{n-1}{3}$ and the equality holds only when $G$ is the union of a $K_{n-1}$ and an isolated vertex.

     Now we prove the case when $t \ge 4$.
     By Theorem~\ref{thm: Kruskal-Katona Theorem}, when $k_t(G) \ge \binom{n-1}{t}$, we have $k_3(G) \ge \binom{n-1}{3}$.
     Then by the previous analysis, $G$ must be the union of a $K_{n-1}$ and an isolated vertex and then $k_t(G) \le \binom{n-1}{t}$.

     Suppose $k \geq 3$ and the lemma holds for smaller $k$. 
     Let $v$ be the vertex with maximum degree $\Delta = n-k$.
     Let $U = V(G) \setminus (N(v) \cup \{v\})$, then $|U| = k-1 \ge 2$.
     We claim that there is no edge in $U$. 
     In fact, if there is an edge $u_1 u_2$ contained in $U$, then $(U \setminus \{u_2\}) \cup \{v\}$ forms a dominating set of size $k-1$, a contradiction.
     This means that any two vertices of $U$ cannot be contained in a same $K_t$.
     We then claim that for any two vertices $u_i, u_j \in U$, they have no common neighbor.
     Suppose otherwise, $w$ is a common neighbor of $u_i$ and $u_j$, then $(U \setminus \{u_i, u_j\}) \cup \{v, w\}$ is a dominating set of size $k-1$, a contradiction.
     Let $H$ be the graph obtained by identifying $u_i$ and $u_j$ for any two vertices $u_i, u_j \in U$.
     Now consider a $K_t$ in $G$.
     If $u_i,u_j$ are not in the $K_t$, then it is still a $K_t$ in $H$.
     If one of $u_i, u_j$ is in the $K_t$, then we have a $K_t$ in $H$ by replacing $u_i$ or $u_j$ by the new vertex $w$.
     Since $u_iu_j$ is not an edge, $u_i$ and $u_j$ cannot be both in the $K_t$.
     Then each $K_t$ in $G$ corresponds to a $K_t$ in $H$.
     And since there is no common neighbor of $u_i$ and $u_j$, no two $K_t$ in $G$ correspond to the same $K_t$ in $H$.
     Then $k_t(G) \le k_t(H)$.
     By Lemma~\ref{lem: domination number after identifying} and the induction hypothesis, we have $k_t(G) \le k_t(H) \le \binom{n-k+1}{t}$.
     Moreover, equality holds only when $H$ is the union of a clique $K_{n-k+1}$ and some isolated vertices.
     By analyzing the possible structures of $G$, we can obtain that $G$ must also be the union of a clique $K_{n-k+1}$ and some isolated vertices.
\end{proof}

As a corollary of Lemma~\ref{lem: number of clique}, we have the following result.

\begin{corollary}\label{coro: number of clique}
    Let $k_t(G)$ be the number of copies of $K_t$ in a graph $G$.
    For a graph $G$ with $\phi(G) \le s$ ($s \geq t-1$), then $k_t(G) \le \binom{s+1}{t} $.
    Equality holds when graph $G$ is the union of a clique $K_{s+1}$ and some isolated vertices.
\end{corollary}

% -------------------------------------------------------------
% \section{Structure of $\Tr_r(M_{s+1})$-free hypergraph}\label{sec: structure}
\section{Structure of \texorpdfstring{$\Tr_r (M_{s+1})$}{Tr\_r(M\_\{s+1\})}-free hypergraphs}\label{sec: structure}

If $\mathcal{H}$ contains $M_{s+1}$ as a trace, then there exists a set of pairs of vertices $\{x_iy_i\}_{1 \le i \le s+1}$ such that for each $x_iy_i$, there exists a hyperedge $\{x_i, y_i, z_{i,1}, \dots, z_{i,r-2}\}$ with $z_{i,j} \notin \{x_i, y_i\}_{1 \le i \le s+1}$.
We call $\{x_iy_i\}_{1 \le i \le s+1}$ the \textbf{core} of the trace.

The following two lemmas are devoted to the proof of the lower bound construction.
\begin{lemma}\label{lem: hypergraph bounded dominated number size}
    Let $\mathcal{G}$ be a $r$-uniform hypergraph with $\phi(\mathcal{G}) \le s$ without isolated vertices.
    Then $v(\mathcal{G}) \le 3r(s+1) = O_{s,r}(1)$.
\end{lemma}

\begin{proof}
    We define a directed graph $D$ with vertex set $V(\mathcal{G})$.
    Since $\mathcal{G}$ contains no isolated vertex, for each $u \in V(\mathcal{G})$, there exists a hyperedge $e_u$ in $\mathcal{G}$ containing $u$, then add an arc from $u$ to $v$ if $v \in e_u$.
    Then the maximum outdegree in $D$ is at most $r-1$.
    Hence, the minimum indegree in $D$ is at most $r-1$.
    Suppose otherwise $v(\mathcal{G}) \ge 3r(s+1)$, then there exists an independent set of size at least $s+1$ in $D$ constructed by iteratively picking the vertex with minimum indegree and deleting all its in-neighbors and out-neighbors.
    Let $I$ be the independent set in $D$, then $I$ is a dominated set of size at least $s+1$ in $\mathcal{G}$ by the definition, a contradiction with $\phi(\mathcal{G}) \leq s$.
\end{proof}

\begin{lemma}\label{lem: lower bound general}
    Let $\mathcal{G}$ be a $(r-1)$-uniform hypergraph with $\phi(\mathcal{G}) \le s$ without isolated vertices.
    Let $U$ be the vertex set of $\mathcal{G}$, $|U| = O_{s,r}(1)$.
    Let $\mathcal{H}$ be an $r$-uniform hypergraph on $n$ vertices constructed by letting each vertex $u \notin U$ and a hyperedge in $\mathcal{G}$ be a hyperedge in $\mathcal{H}$.
    Then $\mathcal{H}$ does not contain $M_{s+1}$ as a trace.
\end{lemma}

\begin{proof}
    Suppose otherwise $\mathcal{H}$ contains a trace of $M_{s+1}$ with core $\{x_1y_1, x_2y_2, \dots, x_{s+1}y_{s+1}\}$.
    Let $X = \{x_i\}_{1 \leq i \leq s+1}$, $Y = \{y_i\}_{1 \leq i \leq s+1}$.
    % Let $e_i = \{x_i, y_i, z_{i,1}, z_{i,2}, \dots, z_{i,r-2}\}$ denote the hyperedge in $\mathcal{H}$ containing $x_iy_i$ where $z_{i,j} \notin  X\cup Y$ for each $1 \le j \le r-2$.
    % Let $Z = \bigcup_{1 \leq i \leq s+1} \{z_{i,1}, z_{i,2}, \dots, z_{i,r-2}\}$.
    Note that for each $i$, at least one of $x_i$ and $y_i$ belongs to $U$ by the construction. 
    We may assume each $x_i$ is in $U$.
    Then for each $x_i$, let $e_i$ be the hyperedge containing $x_i,y_i$ and intersecting $X \cup Y$ at exactly $x_i$ and $y_i$.
    By the construction, there exists a hyperedge in $\mathcal{G}$ containing $x_i$ and intersecting $X$ at exactly $x_i$, denoted by $S_i$.
    Then $X$ is a dominated set with size $s+1$ in $\mathcal{G}$, since for each $x_i$, there exists a hyperedge $S_i$ intersecting $X$ at exactly $x_i$, a contradiction.
\end{proof}

Let $\mathcal{H}$ be an $r$-uniform hypergraph on $n$ vertices forbidding $M_{s+1}$ as a trace with the maximum number of hyperedges.
We say a $(r-1)$-set is \textbf{heavy} if it is contained in at least $r(s+1)+1$ hyperedges.
Otherwise, we say it is \textbf{light}.
Let $\mathcal{H}_1$ be the set of hyperedges in which all $(r-1)$-sets are light.
Let $\mathcal{H}_2$ be the set of hyperedges with at least one heavy $(r-1)$-set.
Then $\mathcal{H}_1$ and $\mathcal{H}_2$ are a partition of $E(\mathcal{H})$.
% Let $\mathcal{G}_1$ be the $(r-1)$-uniform hypergraph obtained by picking all light $(r-1)$-sets from each hyperedge in $\mathcal{H}_1$, and 
Let $\mathcal{G}_2$ be the $(r-1)$-uniform hypergraph obtained by picking all heavy $(r-1)$-sets as hyperedges.

\begin{claim}\label{claim: size of H1}
	$|\mathcal{H}_1| = O_{s,r}(1)$.
\end{claim}

\begin{proof}
    Let $G_1$ be the $2$-shadow graph of $\mathcal{H}_1$.
    That is, the graph obtained by picking all pairs contained in some hyperedge in $\mathcal{H}_1$.
    First, we claim that $G_1$ is $M_{3r(s+1)}$-free.
    Suppose otherwise, there exists a matching $\{x_i y_i\}_{1 \leq i \leq 3r(s+1)}$ in $G_1$.
    Then we construct a directed graph $D$ with vertex set $\{1,2, \ldots, 3r(s+1)\}$.
    And for each $i$, there exists a hyperedge $e_i$ in $\mathcal{G}_1$ containing $x_i$ and $y_i$, then add an arc from $i$ to $j$ if $x_j \in e_i$ or $y_j \in e_i$.
    Then the maximum outdegree of a vertex in $D$ is at most $r-2$.
    By iteratively picking the vertex with minimum indegree and deleting all its in-neighbors and out-neighbors, we can get an independent set of size at least $s+1$ in $D$, denoted by $I$.
    Then we have a trace of $M_{s+1}$ with core $\{x_i y_i\}_{i \in I}$, a contradiction.

    Then by Tutte-Berge~\cite{berge1958couplage} formula, there exists a set $U$ with size at most $9r(s+1)$ such that $G_1 - U$ is independent.
    Note that every hyperedge in $\mathcal{H}_1$ corresponds to a $K_r$ in $G_1$.
    If a $K_r$ contains a vertex outside $U$, then the rest $r-1$ vertices must be in $U$.
    There are at most $\binom{9r(s+1)}{r-1} \cdot r(s+1)$ such hyperedges since the $(r-1)$-set contained in $U$ has to be light by the definition of $\mathcal{H}_1$.
    Then we have $|\mathcal{H}_1| \le \binom{9r(s+1)}{r-1} \cdot r(s+1) + \binom{9r(s+1)}{r} = O_{s,r}(1)$.
\end{proof}

\begin{claim}\label{claim: phi of G2}
    $\phi(\mathcal{G}_2) \leq s$.
\end{claim}

\begin{proof}
    Suppose otherwise, there exists a dominated set $X$ with $|X| = s+1$.
    Let $X = \{x_1, x_2, \dots, x_{s+1}\}$.
    For each $x_i$, there exists a heavy $(r-1)$-set $\{x_i, y_{i,1}, y_{i,2}, \dots, y_{i, r-2}\}$ and $\{y_{i,1}, y_{i,2}, \dots, y_{i, r-2}\}\subseteq V(\mathcal{G}_2) \setminus X$.
    Let $Y = \cup_{i=1}^{s+1} \{y_{i,1}, y_{i,2}, \dots, y_{i, r-2}\}$, $|Y| \le (r-2)(s+1)$.
    By the definition of heavy $(r-1)$-set, there exists a hyperedge $\{z_1, x_1, y_{1,1}, y_{1,2}, \dots, y_{1, r-2}\}$ such that $z_1 \notin X \cup Y$.
    For each heavy $(r-1)$-set $\{x_j, y_{j,1}, y_{j,2}, \dots, y_{j, r-2}\}$, by the definition of heavy $(r-1)$-set, there exists a hyperedge $\{z_j, x_j,  y_{j,1}, y_{j,2}, \dots, y_{j, r-2}\}$ such that $z_j \notin X \cup Y \cup \{z_1,\ldots, z_{j-1}\}$.
    Then it forms a trace of $M_{s+1}$ with core $\{z_1 x_1, z_2 x_2, \dots, z_{s+1} x_{s+1}\}$, a contradiction.
\end{proof}

% -------------------------------------------------------------
\section{Proofs of main theorems}\label{sec: main proofs}

\subsection{Proofs of Theorem \ref{thm: general r} and Theorem \ref{thm: r=3 case}}

In this subsection, we give the proofs of Theorem~\ref{thm: general r} and Theorem~\ref{thm: r=3 case}.

\begin{proof}[Proof of Theorem \ref{thm: general r}]
    The lower bound follows from Lemma~\ref{lem: hypergraph bounded dominated number size} and Lemma~\ref{lem: lower bound general}.
    Now we prove the upper bound.
    Combining Claim~\ref{claim: size of H1} and Claim~\ref{claim: phi of G2}, we have
    \begin{equation*}
        e(\mathcal{H}) \leq e(\mathcal{H}_1) + e(\mathcal{H}_2) = e(\mathcal{G}_2) \cdot n + O_{s,r}(1) \leq  f_{r-1}(s) \cdot n + O_{s,r}(1)
    \end{equation*}
\end{proof}

\begin{proof}[Proof of Theorem~\ref{thm: r=3 case}]
First we prove the lower bound.
Let $\mathcal{H}$ be the hypergraph constructed in Theorem~\ref{thm: r=3 case} and $G$ be the graph used in the construction.
It is easy to prove the construction has the desired number of hyperedges, then it remains to prove that $\mathcal{H}$ does not contain $M_{s+1}$ as a trace.
If $\mathcal{H}$ contains a trace of matching with size $s+1$, denote the core by $\{x_iy_i\}_{1 \le i \le s+1}$.
Notice that for each $i$, at least one of $x_i$ and $y_i$ is in $V(G)$.
We may assume $x_1, x_2, \ldots, x_{s+1} \in V(G)$.
We claim that the hyperedge inside $V(G)$ can not be used to form a trace.
Suppose the hyperedge containing $x_1y_1$ used in the trace, denoted by $e_1$,  is contained in $V(G)$, then $x_2, x_3, \ldots, x_{s+1} \in V(G) \setminus e_1$. 
Note that $|V(G)\setminus e_1| = s-1$, a contradiction.
Since the hyperedges contained in $V(G)$ can never be used, we can prove the result by Lemma~\ref{lem: lower bound general}.
Then we focus on the upper bound.

Let $\mathcal{H}$ be a $3$-uniform hypergraph on $n$ vertices forbidding $M_{s+1}$ as a trace with the maximum number of hyperedges.
Let $\mathcal{H}_1, \mathcal{H}_2$ be the sets of hyperedges defined in Section~\ref{sec: structure} and let $\mathcal{G}_2$ be the hypergraph defined in Section~\ref{sec: structure}.
When $r=3$, $\mathcal{G}_2$ is a graph, denoted by $G_2$ in the following.

Combining Claim~\ref{claim: size of H1} and Claim~\ref{claim: phi of G2} and the lower bound construction, we have 
\begin{equation}\label{eq: H size}
	\left\lfloor \frac{s(s+2)}{2} \right\rfloor (n-s-2) + \binom{s+2}{3} \le e(\mathcal{H}) \le e(\mathcal{H}_1) + e(\mathcal{H}_2) \leq O_{s,r}(1) + e(G_2)\cdot n.
\end{equation}
When $n$ is sufficiently large, we have $e(G_2) \geq \lfloor s(s+2)/2 \rfloor$.
By Claim~\ref{claim: phi of G2} and Corollary~\ref{cor: vizing and phi}, we have that $G_2$ is exactly the graph obtained by removing a minimum edge cover from the complete graph $K_{s+2}$ since $G_2$ has no isolated vertices.
Let $U$ be the vertex set of $G_2$, $|U| = s+2$. 
We prove the following claim:

\begin{claim}\label{claim: analysis of rest hyperedeges}
	Each hyperedge in $\mathcal{H}$ is either contained in $U$, or the union of a vertex outside $U$ and an edge in $G_2$.
\end{claim}

\begin{proof}
	Suppose otherwise, there exists a hyperedge $e$ in $\mathcal{H}$ which violates the claim.
	Then $e$ is one of the following:
	\begin{enumerate}
		\item $e$ is contained in $V(\mathcal{H})\setminus U$;
		\item $e$ contains two vertices outside $U$ and one vertex in $U$.
		\item $e$ contains one vertex outside $U$ and two vertices in $U$, but the two vertices in $U$ are not adjacent in $G_2$.
	\end{enumerate}

	For case 1, assume $e = w_1w_2w_3$.
	Let $G'$ be the graph obtained by adding $w_1w_2$ to $G_2$.
	Then $G'$ contains a dominated set of size $s+1$, say $X = \{x_1, x_2, \ldots, x_{s+1}\}$.
	For each $x_i$, there exists a neighbor $y_i$ outside $X$ in $G'$.
	Claim that $w_1w_2$ must be contained in $\{x_iy_i\}_{1 \le i \le s+1}$.
	Otherwise, it would violate that $\phi(G_2) \le s$.
	We may assume $w_1w_2 = x_1y_1$.
	Then there exists a hyperedge $x_1y_1z_1$ where $z_1 = w_3 \notin X \cup Y$.
	For $x_iy_i$, $i \ge 2$, since it is a heavy pair, there exists a hyperedge $x_iy_iz_i$ where $z_i \notin X \cup Y \cup \{z_1, \ldots, z_{i-1}\}$.
	It would result in a trace of $M_{s+1}$ with core $\{z_1x_1, z_2x_2, \ldots, z_{s+1}x_{s+1}\}$, a contradiction.

	For case 2, assume $e = uw_1w_2$ where $u \in U$ and $w_1, w_2 \notin U$.
	Similarly, let $G'$ be the graph obtained by adding $uw_1$ to $G_2$ and then we can conclude by a similar argument as in Case 1.

	For case 3, assume $e = wu_1u_2$ where $u_1,u_2 \in U$, $w \notin U$ and $u_1u_2 \notin E(G_2)$.
	Similarly, let $G'$ be the graph obtained by adding $u_1u_2$ to $G_2$ and then we can conclude by a similar argument as in Case 1.
\end{proof}

By Claim~\ref{claim: analysis of rest hyperedeges}, we have 
\begin{equation}\label{eq: hyperedeges}
	e(\mathcal{H}) \le \left\lfloor \frac{s(s+2)}{2} \right\rfloor (n-s-2) + \binom{s+2}{3},
\end{equation}

and equality holds only when $\mathcal{H}$ is the hypergraph constructed in Theorem~\ref{thm: r=3 case}.
\end{proof}

\subsection{Proof of Theorem \ref{thm: generalized Turan number} and~\ref{thm: generalized Turan number r=3}}
\begin{proof}[Proof of Theorem \ref{thm: generalized Turan number}]
    First we prove the lower bound.
Let $\mathcal{H}$ be the hypergraph constructed in Theorem~\ref{thm: generalized Turan number} and $\mathcal{G}$ be the hypergraph used in the construction.
Suppose otherwise, there exists a trace of $M_{s+1}$ with core $\{x_1y_1, x_2y_2, \ldots, x_{s+1}y_{s+1}\}$.
Let $X = \{x_i\}_{1 \le i \le s+1}$.
By the construction, for each $i$, at least one of $x_i, y_i$ is in $U$.
We may assume $x_1, x_2, \ldots, x_{s+1} \in U$.
Note that each edge in $\mathcal{H}$ contains a $(r-1)$-set which is a hyperedge in $\mathcal{G}$ and if the hyperedge is contained in $U$, then every $(r-1)$-set is a hyperedge in $\mathcal{G}$.
By the definition of trace, for each $i$ there exists a hyperedge $e_i$ in $\mathcal{H}$ containing $x_iy_i$ and intersecting $X$ at exactly $x_i$.
Then there exists a $(r-1)$-set $s_i \subseteq e_i$ in $\mathcal{G}$ containing $x_i$ and intersecting $X$ at exactly $x_i$.
Then $X$ is a dominated set with size $s+1$ in $\mathcal{G}$, a contradiction.

    Now we prove the upper bound.
    Let $\mathcal{H}$ be the hypergraph on $n$ vertices forbidding $M_{s+1}$ as a trace with the maximum number of copies of $\mathcal{K}^r_{t}$.
    Let $\mathcal{H}_1, \mathcal{H}_2$ be the sets of hyperedges defined in Section~\ref{sec: structure} and let $\mathcal{G}_2$ be the hypergraph defined in Section~\ref{sec: structure}.
    Let $U$ denote the vertex set of $\mathcal{G}_2$.
    By Lemma~\ref{lem: hypergraph bounded dominated number size}, $|U| = O_{s,r}(1)$.

    First, if a copy of $\mathcal{K}^r_{t}$ contains a hyperedge in $\mathcal{H}_1$, we claim that the number of such copies is $O_{s,r}(1)$.
    Assume $e_1 \in \mathcal{H}_1$ is in a copy of $\mathcal{K}^r_{t}$.
    Pick a vertx $v \in e_1$, then $e_1\setminus \{v\}$ is light by the definition of $\mathcal{H}_1$.
    For each vertex $u \notin e_1$ from the $\mathcal{K}_t^r$, $(e_1 \setminus \{v\}) \cup \{u\}$ is a hyperedge.
    Then there are at most $r(s+1)$ such $u$ by the definition of light $(r-1)$-set.
    Combining Claim~\ref{claim: size of H1}, we have that the number of copies of $\mathcal{K}^r_{t}$ containing a hyperedge in $\mathcal{H}_1$ is $O_{s,r}(1)$.

    Then, the number of copies of $\mathcal{K}^r_{t}$ contained in $U$ is also $O_{s,r}(1)$ since $|U| = O_{s,r}(1)$.
    If a copy of $\mathcal{K}^r_t$ is not contained in $U$ and contains no hyperedge in $\mathcal{H}_1$, then we claim that there is exactly one vertex in the copy of $\mathcal{K}^r_t$ that is not contained in $U$.
    Suppose otherwise, let $u,v$ be two vertices in the copy of $\mathcal{K}^r_t$ that are not contained in $U$.
    Then pick any hyperedge $e$ containing $u,v$ in the copy of $\mathcal{K}^r_t$.
    Since the copy of $\mathcal{K}^r_t$ contains no hyperedge in $\mathcal{H}_1$, we have $e \in \mathcal{H}_2$, and there exists a heavy $(r-1)$-set in $e$.
    At least one of $u,v$ should be in the heavy $(r-1)$-set, and thus contained in $U$ by the definition of $\mathcal{G}_2$.

    Now we count the number of such copies of $\mathcal{K}^r_t$.
    Assume $v$ is the vertex not contained in $U$, then pick any $(r-1)$-set $S$ in $\mathcal{K}^r_t$ not containing $v$.
    Since the hyperedge $S\cup\{v\}$ is in $\mathcal{H}_2$ and $v \notin U$, the set $S$ must be heavy. 
    Then the $\mathcal{K}_t^r$ minus $v$ corresponds to a $\mathcal{K}_{r-1}^{t-1}$ in $\mathcal{G}_2$.
    Then by Claim~\ref{claim: phi of G2} and definition of $g_r(s,t)$, the number of such copies of $\mathcal{K}^r_t$ is at most $g_{r-1}(s, t-1) \cdot n$.

    As a conclusion, the number of copies of $\mathcal{K}^r_t$ in $\mathcal{H}$ is at most $g_{r-1}(s, t-1) \cdot n + O_{s,r}(1)$.
\end{proof}

\begin{proof}[Proof of Theorem~\ref{thm: generalized Turan number r=3}]
    Corollary~\ref{coro: number of clique} actually proves $g_{2}(s,t-1) = \binom{s+1}{t-1}$ for $s \ge t-2$.
    Clearly, when $s < t-2$, $g_{2}(s,t-1) = 0$.
    The remaining proof is devoted to solve the constant term.
    The lower bound construction can be verified by a similar argument as in the proof of Theorem~\ref{thm: r=3 case}.
    The details are omitted here.

    Let $\mathcal{H}$ be the hypergraph on $n$ vertices forbidding $M_{s+1}$ as a trace with the maximum number of copies of $K^3_{t}$.
    Let $\mathcal{H}_1, \mathcal{H}_2$ be the sets of hyperedges defined in Section~\ref{sec: structure} and let $\mathcal{G}_2$ be the hypergraph defined in Section~\ref{sec: structure}.
    When $r=3$, $\mathcal{G}_2$ is a graph, denoted by $G_2$ in the following.
    Let $U$ be the vertex set of $G_2$.

    By the proof of Theorem~\ref{thm: generalized Turan number}, the number of copies of $\mathcal{K}^3_{t}$ in $\mathcal{H}$ is at most $k_{t-1}(G_2) \cdot n + O_{s,r}(1)$.
    From the lower bound construction, we have $k^3_t(\mathcal{H}) \ge \binom{s+1}{t-1} \cdot (n-s-1) + \binom{s+1}{t}$.
    When $n$ is sufficiently large, we have $k_{t-1}(G_2) \ge \binom{s+1}{t-1}$.
    Then apply Corollary~\ref{coro: number of clique}, we have $k_{t-1}(G_2) = \binom{s+1}{t-1}$ and $G_2$ is a complete graph on $s+1$ vertices.
    By a similar argument as Claim~\ref{claim: analysis of rest hyperedeges}, we have that each hyperedge in $\mathcal{H}$ is either contained in $U$, or the union of a vertex outside $U$ and an edge in $G_2$.
    Then the number of copies of $\mathcal{K}^3_{t}$ in $\mathcal{H}$ is at most $\binom{s+1}{t-1} \cdot (n-s-1) + \binom{s+1}{t}$.
\end{proof}

\subsection{Proof of Theorem~\ref{thm: matching and a graph}}

\begin{proof}[Proof of Theorem~\ref{thm: matching and a graph}]
    We first prove the lower bound.
    Let $\mathcal{H}$ be the hypergraph defined in Theorem~\ref{thm: matching and a graph} and $\mathcal{G}$ be the hypergraph used in the definition of $\mathcal{H}$.
    Let $U$ be the vertex set of $\mathcal{G}$.
    By Lemma~\ref{lem: lower bound general}, $\mathcal{H}$ does not contain $M_{s+1}$ as a trace.
    Then it remains to prove that $\mathcal{H}$ does not contain $F$ as a trace.

    Suppose otherwise, there exists a trace of $F$ on the vertex set $S$ in $\mathcal{H}$, that is, there exists a mapping $\lambda$ from $V(F)$ to $S$ such that for every edge $uv \in E(F)$, there exists a hyperedge in $\mathcal{H}$ intersecting $S$ at exactly two vertices $\lambda(u)$ and $\lambda(v)$.
    First, by the construction, $\{v \in V(F): \lambda(v) \notin U\}$ is an independent set in $F$, denoted by $I$.
    Then we claim that there is a dominated copy of $(F-I,N(I))$ in $\mathcal{G}$ on $S \cap U$.
    First, for each $uv \in E(F)$ and $u,v \notin I$, there exists a hyperedge $e_{uv}$ in $\mathcal{H}$ intersecting $S$ at exactly two vertices $\lambda(u)$ and $\lambda(v)$.
    By the construction, there exists a $(r-1)$-hyperedge $s_{uv} \subseteq e_{uv}$ in $\mathcal{G}$ such that $s_{uv} \cap S \cap U = \{\lambda(u), \lambda(v)\}$.
    Then for each $uv \in E(F)$ with $u\notin I$ and $v \in I$, there exists a hyperedge $e_{uv_j}$ in $\mathcal{H}$ intersecting $S$ at exactly two vertices $\lambda(u)$ and $\lambda(v)$.
    Also, there exists a $(r-1)$-hyperedge $s_{uv} \subseteq e_{uv_j}$ in $\mathcal{G}$.
    Since $v \in I$, we have $s_{uv} = e_{uv_j} \setminus \{\lambda(v)\}$.
    Then $s_{uv}$ is a hyperedge in $\mathcal{G}$ intersecting $S \cap U$ at exactly $\lambda(u)$.
    From definition, there is a dominated copy of $(F-I,N(I))$ in $\mathcal{G}$ on $S \cap U$, a contradiction.

    Now we prove the upper bound.
    In the proof of this theorem, we modify the definition of heavy $(r-1)$-set and light $(r-1)$-set.
    We say a $(r-1)$-set is heavy if it is contained in at least $3r(s+1)+v(F)$ hyperedges.
    Otherwise, we say it is light.
    Let $\mathcal{H}$ be the hypergraph on $n$ vertices forbidding $M_{s+1}$ as a trace and $F$ as a trace with the maximum number of hyperedges.
    Let $\mathcal{H}_1, \mathcal{H}_2$ be the partition of $E(\mathcal{H})$, and $\mathcal{G}_2$ be the hypergraph defined in Section~\ref{sec: structure}.
    Let $U$ be the vertex set of $\mathcal{G}_2$.
    By Lemma~\ref{lem: hypergraph bounded dominated number size}, $|U| = O_{s,r}(1)$.
    With a modification of the proof of Claim~\ref{claim: size of H1} and Claim~\ref{claim: phi of G2}, we have that $|\mathcal{H}_1| = O_{s,r,v(F)}(1)$ and $\phi(\mathcal{G}_2) \le s$.

    Since $|U| = O_{s,r}(1)$, the number of hyperedges in $\mathcal{H}_2$ contained in $U$ is also $O_{s,r}(1)$.
    Then we claim that if $e$ is a hyperedge in $\mathcal{H}_2$ containing at least one vertex outside $U$, then $e$ must be the union of a vertex outside $U$ and a hyperedge in $\mathcal{G}_2$.
    First, $e$ contains a heavy $(r-1)$-set, thus, $e$ contains at most one vertex outside $U$.
    Then let $v$ be the vertex in $e \setminus U$, then $e\setminus \{v\}$ is heavy and thus $e$ is the union of a vertex~($v$) outside $U$ and a hyperedge~($e\setminus \{v\}$) in $\mathcal{G}_2$.

    Let $W$ be the collection of vertices in $V(\mathcal{H})\setminus U$ such that the number of hyperedges in $\mathcal{H}_2$ containing $w$ is at least $h_{r-1}(s, F) + 1$.
    % We would like to show that $|W| = O_{r,s,v(F)}(1)$.

    \begin{claim}\label{claim: upper bound of W}
        $|W| \le v(F) \cdot \binom{|U|^2}{h_{r-1}(s, F)+1} = O_{s,r,v(F)}(1)$.
    \end{claim}
    \noindent
    \textbf{Proof of Claim~\ref{claim: upper bound of W}}:
    For any $w \in W$, let $\mathcal{G}_w$ be the link graph of $w$ in $\mathcal{G}_2$, that is, the graph of heavy $(r-1)$-sets such that there exists a hyperedge in $\mathcal{H}_2$ containing $w$ and the heavy $(r-1)$-set.

    Suppose otherwise, then by the pigeonhole principle, there exists $w_1,w_2,\ldots, w_{v(F)}$ such that the intersection of $\mathcal{G}_{w_1}, \mathcal{G}_{w_2}, \ldots, \mathcal{G}_{w_{v(F)}}$ has at least $h_{r-1}(s, F) + 1$ hyperedges.
    Let $\mathcal{G}_w$ be the intersection of $\mathcal{G}_{w_1}, \mathcal{G}_{w_2}, \ldots, \mathcal{G}_{w_{v(F)}}$.
    By the definition of $W$, $\mathcal{G}_w$ is a subgraph of $\mathcal{G}_2$ and contains at least $h_{r-1}(s, F) + 1$ hyperedges.
    % Then by the definition of $h_{r-1}(s, F)$, $\mathcal{G}_w$ contains a member of $\mathcal{D}(F)$ as a trace.
    Then by the definition, there is a dominated copy of $(F-I,N(I))$ in $\mathcal{G}_w$ for some independent set $I$ in $F$.
    Let $F'$ be the graph $F-I$ and $I = \{x_1, x_2, \ldots, x_k\}$, then there exists a mapping $\tau$ from $V(F')$ to a vertex set $S$ in $V(\mathcal{G}_w)$ satisfying the conditions in the definition of the dominated copy.
    In the following paragraph, when we say the graph $F'$, we mean the image of $F'$ under $\tau$ on $S$ in $\mathcal{G}_w$.
    We are going to find a trace of $F'$ on the vertex set $S' = S \cup \{w_1, w_2, \ldots, w_{v(F)}\}$, where $w_j$ plays the role of $v_j$ in $I$ for each $j$.

    % Let $F'$ be the member of $\mathcal{D}(F)$ and $I=\{x_1, x_2, \ldots, x_k\}$ be the independent set in $F$ such that $F' = F - I$.
    % Then there exists a vertex set $S$ with $|S| = v(F')$ such that $\{e\cap S : e \in E(\mathcal{G}_w)\}$ contains $F'$ as a subgraph.
    % In the following, when we say the graph $F'$, we mean the graph $F'$ on $S$ in $\{e\cap S : e \in E(\mathcal{G}_w)\}$.
    % We are going to find a trace of $F$ on the vertex set $S' = S \cup \{w_1, w_2, \ldots, w_k\}$, where $w_j$ plays the role of $v_j$ in $I$ for each $j$. 
    
    Suppose $uv$ is an edge in $F$.
    If $u,v \in V(F')$, then there exists a $(r-1)$-hyperedge $s_{uv}$ in $\mathcal{G}_w$ such that $s_{uv} \cap S = \{u,v\}$.
    By the definition of $\mathcal{G}_2$, $s_{uv}$ is a heavy $(r-1)$-set.
    Thus there exists a hyperedge $s_{uv} \cup \{z_{uv}\}$ such that $z_{uv} \notin S \cup \{w_1, w_2, \ldots, w_k\}$.
    It is impossible that both $u$ and $v$ are outside $V(F')$ since $I$ is an independent set.
    The only remaining case is that one of $u,v$ is in $V(F')$ and the other one is in $I$.
    Assume $u \in V(F')$ and $v = v_j \in I$.
    Then there exists a hyperedge $s_{uv_j}$ in $\mathcal{G}_w$ such that $s_{uv_j} \cap S = \{u\}$.
    % Since we have a copy of $F'$ in $\mathcal{G}_w$ on $S$, there exists a hyperedge $s_{uv_j}$ in $\mathcal{G}_w$ such that $s_{uv_j} \cap S = \{u,v_j\}$.
    Recall that $\mathcal{G}_w$ is a subgraph of $\mathcal{G}_{w_j}$, then $s_{uv_j} \cup \{w_j\}$ is a hyperedge in $\mathcal{H}$.
    Note that $s_{uv_j} \cup \{w_j\}$ intersects $S'$ at exactly two vertices $u$ and $w_j$.
    As a result, for every edge in $F$, there exists a hyperedge in $\mathcal{H}$ intersecting $S'$ at exactly two endpoints of the edge.
    It forms a trace of $F$, a contradiction.
    \hfill $\blacksquare$ \par
    \vspace*{.2cm}

    Since $|W| = O_{r,s,v(F)}(1)$, the number of hyperedges in $\mathcal{H}_2$ containing a vertex outside $U$ is at most $|W| \cdot |U|^r + h_{r-1}(s, F) \cdot n$.
    Combining the number of hyperedges in $\mathcal{H}_1$ and $\mathcal{H}_2$ contained in $U$, we have that the number of hyperedges in $\mathcal{H}$ is at most $ h_{r-1}(s, F) \cdot n + O_{r,s,v(F)}(1)$.
\end{proof}

% -------------------------------------------------------------
\section{Concluding remarks}\label{sec: concluding remarks}

To determine the coefficient $f_{r-1}(s)$ in Theorem~\ref{thm: general r} is a natural question.
We have the following conjecture.
\begin{conjecture}
    $f_r(s)$ achieves the maximum for either $\mathcal{K}^{r}_{s+r-1}$, or the hypergraph $\mathcal{G}$ on $s+r$ vertices obtained by removing a minimum edge set such that every $(r-1)$-set is covered at least once from the complete hypergraph $\mathcal{K}_{s+r}^{r}$.
    Especially, if the Steiner system $S(s+r, r, r-1)$ exists, then $f_r(s) = \binom{s+r}{r}- \binom{s+r}{r-1}/r$ in this case~(The Steiner system $S(s+r, r, r-1)$ is a set of $r$-element subsets of an $(s+r)$-element set such that each $(r-1)$-subset is contained in exactly one $r$-subset).
\end{conjecture}

We have a conjecture for $g_r(s,t)$ too.
\begin{conjecture}
    $g_r(s,t)$ achieves the maximum for the complete $r$-uniform hypergraph $\mathcal{G}$ on $s+r-1$ vertices.
\end{conjecture}

It is worth mentioning that, once we know the extremal structure for $f_r(s)$ and $g_r(s,t)$, we can determine the constant term in Theorem~\ref{thm: general r} and Theorem~\ref{thm: generalized Turan number} by a similar argument as in the proof of Theorem~\ref{thm: r=3 case} and Theorem~\ref{thm: generalized Turan number r=3}.

% -------------------------------------------------------------
\section*{Acknowledgement}
% M. Lu is supported by the National Natural Science Foundation of China (Grant 12171272 \& 12161141003).
% This research was supported by the National Natural Science Foundation of China (Grant 12171272 \&12161141003).
The research of Wang is supported by the China Scholarship Council (No. 202506210200) and the National Natural Science Foundation of China (Grant 12571372).

\noindent
The research of Cheng is supported by the National Natural Science Foundation of China (Nos. 12131013 and 12471334), Shaanxi Fundamental Science Research Project for Mathematics and Physics (No. 22JSZ009) and the China Scholarship Council (No. 202406290241).

\noindent
The research of Zhao is supported by the China Scholarship Council (No. 202506210250) and the National Natural Science Foundation of China (Grant 12571372).

% ------------ bib part ---------------------
\bibliography{ref}
\bibliographystyle{wyc4}

\end{document}